\documentclass[a4paper, reqno]{amsart}

\usepackage[latin1]{inputenc}
\usepackage[T1]{fontenc}
\usepackage[english]{babel}
\usepackage{amsmath, amsthm, amssymb, amsfonts}

\usepackage{mathrsfs}
\usepackage{physics}
\usepackage{mathtools}
\usepackage{array}
\usepackage{url}
\usepackage{cite}
\usepackage{enumitem}
\usepackage{etoolbox}
\usepackage{graphicx}
\usepackage{booktabs}

\usepackage{xcolor} 
\usepackage{scrextend}
\usepackage{fancyhdr}
\usepackage{graphicx}
\usepackage[colorlinks, citecolor=red, urlcolor=blue]{hyperref}

\usepackage{pgfplots}
\pgfplotsset{width=9cm,compat=1.18}
\usepgfplotslibrary{fillbetween,external}
\usetikzlibrary{intersections}

\allowdisplaybreaks

\theoremstyle{remark}
\newtheorem{remark}{Remark}
\theoremstyle{plain}
\newtheorem{theorem}{Theorem}
\newtheorem{proposition}{Proposition}
\newtheorem{lemma}{Lemma}

\begin{document}
	
\title[Boundary stabilization of an EBB with axial force and internal delay]{Boundary stabilization of an Euler-Bernoulli beam with axial force and internal delay}

\author{Ben Bakary Junior Siriki}
\address{Universite Nangui Abrogoua, Abidjan, C\^ote d'Ivoire}
\email{benbjsiriki@gmail.com}

\author{Adama Coulibaly}
\address{Universite F\'elix Houphou\"et-Boigny, Abidjan, C\^ote d'Ivoire}
\email{couliba@yahoo.fr}


\begin{abstract}
This paper investigates the boundary stabilization of an Euler-Bernoulli beam under constant axial tension and subject to an internal time-delay. First, the well-posedness of the system is established using semigroup of linear operators theory. Then, by constructing an adequate Lyapunov functional, we establish the exponential stability of the system within a region of the $(\alpha, \xi)$-plane, explicitly characterized by a system of inequalities involving the physical parameters. This work extends the results of Han et al. (Journal of Difference Equations, 2016) by incorporating the effect of axial tension and an internal delay term without sign constraints.
\end{abstract}	


	
\maketitle

\section{Introduction}

	In this paper, we study the stabilization of a uniform Euler-Bernoulli beam, clamped at one end and subject to a boundary velocity control at its free end. We assume that the beam is under a constant axial tension $T > 0$ and includes an internal delay term. The equations governing the vibratory motion of this structure are given by:
\begin{equation} \label{eq:p1}
	\left\{
	\begin{array}{>{\displaystyle}l}
	y_{tt}(x,t) + y_{xxxx}(x,t) - T y_{xx}(x,t) + \alpha y_t(x, t-\tau) = 0, \\
	\qquad \qquad \qquad \qquad \qquad \qquad \qquad \qquad \qquad \qquad (x,t) \in (0,\ell) \times (0,+\infty), \\
	y(0,t) = y_x(0,t) = 0, \quad t \in (0,+\infty),  \\
	y_{xx}(\ell,t) = 0, \quad t \in (0,+\infty),  \\
  y_{xxx}(\ell,t) -T y_x(\ell,t) = \kappa y_t(\ell,t), \quad t \in (0,+\infty), \\
	y(x,0) = y_0(x), \, y_t(x,0) = y_1(x), \quad x \in [0, \ell], \\
	y_t(x,s) = f_0(x,s), \quad x \in (0, \ell), \, s \in (-\tau, 0),
	\end{array}
	\right.
\end{equation}
where $y(x,t)$ represents the deflection of the beam at position $x$ and time $t$, while $\tau > 0$ denotes the time delay, assumed to be constant. The parameter $\alpha \in \mathbb{R}$ is the coefficient associated with the delay term, and $\kappa > 0$ represents the gain of the boundary control at $x=\ell$. Finally, the functions $y_0, y_1$, and $f_0$ define the initial data of the system.

The stabilization of Euler-Bernoulli beams constitutes a pivotal issue in control theory and continues to attract significant attention owing to its wide-ranging technological applications (see \cite{du2022, kundu2017, chakravarthy2010, turner2001, luo1993, balakrishnan1986}). 
To address the precision requirements of modern engineering systems, various stabilization strategies have been explored in recent decades. 
Among them, the contributions of Guo et al. \cite{guo2001} and Dadfarnia et al. \cite{dadfarnia2004} stand out as representative of two distinct and fundamental approaches in the literature, both applied to uniform flexible beams with a tip mass. 
To achieve exponential stabilization, Guo et al. \cite{guo2001} employ a frequency-domain approach based on the asymptotic expansion of generalized eigenvalues and eigenvectors of the linear operator associated with the system, as well as the Riesz basis property. For further contributions based on this approach, see for instance \cite{benamara2025, benamara2025a, teya2023, nicaise2016, wang2013, shang2012}. Conversely, Dadfarnia et al. \cite{dadfarnia2004} favor a time-domain approach based on the direct Lyapunov method. For more details on the application of this method to beam stabilization, we refer the reader to \cite{akil2025, ledkim2023, camasta2024, siriki2025, siriki2025a, guo2022, camasta2025, parada2022, feng2021, han2016}. 

However, the question of stabilization remains largely unexplored for beams models incorporating axial forces (see \cite{siriki2025, siriki2025a, benamara2025, benamara2025a, ledkim2023}), in comparison to the abundant literature dealing with beams in the absence of such constraints (see \cite{boukhari2024, benterki2022, wang2020, teya2023, guo2001, akil2025, camasta2024, shang2012}). 
Accounting for such forces is nevertheless vital, as they can profoundly alter the system's dynamics, potentially triggering to instability.

Furthermore, the study of time-delay effects on dynamical systems remains a subject of great interest within the research community, as they allow for more realistic mathematical models and meet numerous practical requirements. Inherent to any realistic physical system, these delays can cause system destabilization, even when they are arbitrarily small (see \cite{datko1991, datko1993}). To address this issue, several approaches have been developed depending on the location of the delay within the system: one may consult, for instance, \cite{siriki2025, siriki2025a, guo2022, nicaise2016, wang2013, shang2012} for input delays, \cite{yang2019, jilong2017, yang2015} for output delays, and finally \cite{ camasta2025, parada2022, feng2021, akil2021, han2016,  ammari2010} for internal delays. 

However, to the best of our knowledge, the question of exponential stabilization of an Euler-Bernoulli beam, simultaneously subjected to an axial force and a constant internal delay, remains an open problem. 
Our work fits into this framework and aims to extend the stability results obtained by Han et al. \cite{han2016}. More precisely, we treat the case of a an Euler-Bernoulli beam under constant axial tension, in which the coefficient $\alpha$ associated with the delayed velocity term is not subject to any sign constraint. 
Our strategy consists, first, in formulating the closed-loop system \eqref{eq:p1} as a linear evolution problem in an appropriate state space based on Hilbert spaces. The L\"umer-Phillips theorem guarantees that the system is well-posed within the framework of semigroup theory \cite{engel2000, pazy1983} . 
Next, to establish the exponential stability of the model, we employ a Lyapunov-based approach, building upon the framework developed in \cite{feng2021, ammari2010}. Specifically, we construct a suitable Lyapunov functional, which we prove to be equivalent to the energy of the system \eqref{eq:p1}. Finally, a rigorous estimation of its time derivative allows us to derive sufficient conditions guaranteeing the exponential decay of the energy within an explicitly described $(\alpha, \xi)-$region. 

The remainder of this article is organized as follows: Section 2 addresses the well-posedness of system \eqref{eq:p1} using the theory of linear operator semigroups. Subsequently, Section 3 is devoted to the proof of the exponential stability of the system, relying on the construction of an appropriate Lyapunov functional. Finally, Section 4 concludes this work by suggesting several avenues for further research.

\section{Well-posedness of the system: a semigroup approach}

In this section, we will study the existence, uniqueness and regularity of the solution of system \eqref{eq:p1}, using a semigroup approach. First, we will rewrite system \eqref{eq:p1} as an abstract linear Cauchy problem into a tailored Hilbert state space. Next, by using L\"umer-Phillips, we will prove the existence of a unique solution to problem \eqref{eq:p1}.

	\subsection{Abstract Cauchy problem formulation}

To reformulate system \eqref{eq:p1} as an abstract linear Cauchy problem in an appropriate Hilbert space, we first introduce an auxiliary variable. Let
\begin{gather} \label{eq:1}
u(x,s,t) := y_t(x, t-\tau s), \quad (x,s,t) \in (0,\ell) \times (0,1) \times (0, +\infty).
\end{gather}
It follows that $u$ satisfies the system:
\begin{equation} \label{eq:2}
	\left\{
	\begin{array}{>{\displaystyle}l}
	\tau u_t(x,s,t) + u_s(x,s,t) = 0, \quad (x,s,t) \in (0,\ell) \times (0,1) \times (0,+\infty), \\
	u(x,0,t) = y_t(x, t), \quad (x,t) \in (0,\ell) \times (0,+\infty), \\
	u(x,1,t) = y_t(x, t-\tau), \quad (x,t) \in (0,\ell) \times (0,+\infty), \\
	u(x,s,0) = f_0(x, -\tau s), \quad (x,s) \in [0, \ell] \times (0,1).
	\end{array}
	\right.
\end{equation}
Thus, the original system \eqref{eq:p1} is equivalent to the following augmented system:
\begin{equation} \label{eq:p2}
	\left\{
	\begin{array}{>{\displaystyle}l}
	 y_{tt}(x,t) + y_{xxxx}(x,t) - T y_{xx}(x,t) + \alpha u(x,1,t) = 0, \\
	\qquad \qquad \qquad \qquad \qquad \qquad \qquad \qquad \quad (x,t) \in (0,\ell) \times (0,+\infty), \\
	\tau u_t(x,s,t) + u_s(x,s,t) = 0, \quad (x,s,t) \in (0,\ell) \times (0,1) \times (0,+\infty), \\
	y(0,t) = y_x(0,t) = 0, \quad t \in (0, +\infty),  \\
	u(x,0,t) = y_t(x, t), \quad (x,t) \in (0, \ell) \times (0,+\infty), \\
	y_{xx}(\ell,t) = 0, \quad t \in (0, +\infty), \\
	y_{xxx}(\ell,t) - T y_x(\ell,t) = \kappa y_t(\ell,t), \quad t \in (0,+\infty), \\
	u(x,1,t) = y_t(x, t-\tau), \quad (x,t) \in (0, \ell) \times (0,+\infty), \\
	y(x,0) = y_0(x), \, y_t(x,0) = y_1(x), \quad x \in [0, \ell], \\
	u(x,s,0) = f_0(x, -\tau s), \quad (x,s) \in [0, \ell] \times (0,1).
	\end{array}
	\right.
\end{equation}
Next, to set up the appropriate framework for our analysis, we introduce the following functional spaces Let $H^2_E(0,\ell) := \left\{ y \in H^2(0,\ell): \, y(0) = y'(0) = 0 \right\}$. The state space $\mathscr{H}$ is then defined as:
\begin{gather} \label{eq:4}
\mathscr{H} := H^2_E(0,\ell) \times L^2(0,\ell) \times L^2\big((0,\ell) \times (0,1)\big).
\end{gather}
We endow $\mathscr{H}$ with the inner product
\begin{align} \label{eq:5}
\begin{split}
\prec Y_1, Y_2 \succ & := 
\int_{0}^{\ell} v_1(x) \, v_2(x) \, dx + T \int_{0}^{\ell} y_1'(x) \, y_2'(x) \, dx + \int_{0}^{\ell} y_1''(x) \, y_2''(x) \, dx 
\\
& \quad + \, \xi \int_{0}^{\ell} \int_{0}^{1} u_1(x,s) \, u_2(x,s) \, ds \, dx,
\end{split}
\end{align}
for all $Y_i := (y_i, v_i, u_i) \in \mathscr{H}, \, i=1, \, 2$, where $\xi$ a positive constant. With this inner product, $\mathscr{H}$ forms a Hilbert space, and its associated norm is denoted by
$\lVert \cdot \rVert_{\mathscr{H}}$. 
Now, we consider the linear operator $\mathbb{A} : \mathscr{D}(\mathbb{A}) \subset \mathscr{H} \to \mathscr{H}$ defined by:
\begin{gather} \label{eq:6}
\mathscr{D}(\mathbb{A}) := 
\left\{
Y = 
\begin{pmatrix}
y \\ v \\ u
\end{pmatrix} 
\in \mathscr{H} : 
\begin{array}{l}
y \in H^4(0,\ell), \, v \in H^2_E(0,\ell), \\
u \in L^2\left( 0, \ell; \, H^1(0,1) \right), \,v = u(\cdot,0) \\
y''(\ell) = 0, \, y^{(3)}(\ell) -T y'(\ell) = \kappa v(\ell) 
\end{array}
\right\}, \\ \label{eq:6a}
\mathbb{A} Y
:= \begin{pmatrix}v \\ -y^{(4)} + T y'' - \alpha u(\cdot,1) \\ -\tau^{-1} u_s(\cdot,s) \end{pmatrix}.
\end{gather}
Based on the above considerations, system \eqref{eq:p2} is recast as the following abstract Cauchy problem:
\begin{equation} \label{eq:7}
	\left\{
	\renewcommand{\arraystretch}{1.3}
	\begin{array}{>{\displaystyle}l}
	 \dfrac{dY(t)}{dt} = \mathbb{A} Y(t), \; t \in (0, \infty), \\
	 Y(0) = Y_0,
	\end{array}
	\right.
\end{equation}
where $Y(t) := \big( y(\cdot,t), \, v(\cdot,t), \, u(\cdot,t) \big)$ and $Y_0 := \Big( y_0, \, y_1, \, f_0(\cdot,-s\tau) \Big), \; s \in (0, 1)$.

  	\subsection{Existence and uniqueness result}

Before stating the first main result of this paper, we establish some fundamental properties of the operator $\mathbb{A}$ in the following lemmas.	

\begin{lemma} \label{l1}
The linear operator $\mathbb{A}$ given by \eqref{eq:6}--\eqref{eq:6a} is closed and its domain $\mathscr{D}(\mathbb{A})$ is dense in the Hilbert space $\mathscr{H}$.
\end{lemma}

\begin{proof}
First, to establish that the linear operator $\mathbb{A}$ is densely defined in $\mathscr{H}$, we invoke a corollary of the geometric form of the Hahn-Banach theorem \cite[Corollary 1.8 p.47]{haim1983}. Let $Y = (y, v, u) \in \mathscr{H}$ and assume that $\prec Y, (f, g, h) \succ = 0$ for all $(f, g, h) \in \mathscr{D}\big(\mathbb{A}\big)$.  By the definition of the inner product in $\mathscr{H}$, this orthogonality condition explicitly reads:
\begin{equation} \label{eq:8}
\begin{aligned}
& \int_{0}^{\ell} \Big( v(x) g(x) + T y'(x) f'(x) + y''(x) f''(x) \Big) \, dx \\
& + \, \xi \int_{0}^{\ell} \int_{0}^{1} u(s,x) h(x,s) \, ds \, dx = 0.
\end{aligned}
\end{equation}
Choosing $f = g = 0$ and $h \in \mathcal{D}\big( (0,\ell)\times(0,1) \big)$, we find that:
\begin{align} \label{eq:9}
\int_{0}^{\ell} \int_{0}^{1} u(s,x) h(x,s) \, ds \, dx = 0.
\end{align}
It follows that $u = 0$ a.e. in $(0,\ell) \times (0,1)$, which implies $u = 0$ in $L^2\big( (0,\ell)\times(0,1) \big)$ by the density of test functions space $\mathcal{D}\big( (0,\ell)\times(0,1)\big)$ in $L^2\big( (0,\ell)\times(0,1) \big)$. 
Similarly, choosing $f=0, \, h=0$ and $g \in \mathcal{D}(0,\ell)$, the orthogonality condition reduces to:
\begin{equation} \label{eq:9a}
\int_{0}^{\ell} v(x) g(x) \, dx = 0.
\end{equation}
Due to density of $\mathcal{D}(0,\ell)$ in $L^2(0,\ell)$, it follows immediately that $v=0$ in $L^2(0,\ell)$.
In view of the fact that $u = v = 0$, the initial relation \eqref{eq:8} simplifies to:
\begin{gather} \label{eq:10}
\int_{0}^{\ell} \Big( y''(x) f''(x) + T y'(x) f'(x) \Big) \, dx = 0.
\end{gather}
Note that $(f, 0, 0) \in \mathscr{D}(\mathbb{A})$ if and only if $f \in H^2_E(0,\ell) \cap H^4(0,\ell)$.Since this intersection space is dense in $H^2_E(0,\ell)$ (equipped with its natural inner product $\prec \cdot, \cdot \succ_{H^2_E(0,\ell)}$) yields $y=0$. Hence, $Y = 0$ in $\mathscr{H}$.

Next, we show that the operator $\mathbb{A}$ is closed. To this end, let $(Y_n)_{n\in\mathbb{N}} \subset \mathscr{D}(\mathbb{A})$ be a sequence such that $Y_n \to Y$ and $\mathbb{A}Y_n \to F$ in $\mathscr{H}$ as $n \to \infty$. We must show that $Y \in \mathscr{D}(\mathbb{A})$ and $\mathbb{A}Y = F$. More explicitly, setting $Y_n = (y_n, v_n, u_n)$, $Y = (y, v, u)$, and $F = (f_1, f_2, f_3)$, the convergence relations mean that:
\begin{eqnarray} 
\begin{gathered} \label{eq:11}
y_n \underset{n \to \infty}{\longrightarrow} y \text{ in } H^2_E(0,\ell), \; v_n \underset{n \to \infty}{\longrightarrow} v  \text{ in } L^2(0,\ell) \; \text{and } u_n \to u \text{ in } L^2\big((0,\ell) \times (0,1)\big), 
\end{gathered} \\
\begin{gathered} \label{eq:12}
v_n \underset{n \to \infty}{\longrightarrow} f_1 \text{ in } H^2_E(0,\ell), \; -y_n^{(4)} + T y_n'' - \alpha u_n(\cdot, 1) \underset{n \to \infty}{\longrightarrow} f_2 \text{ in } L^2(0,\ell), \\
\text{ and } \tau^{-1} (u_n)_s \underset{n \to \infty}{\longrightarrow} -f_3  \text{ in } L^2\big((0,\ell) \times (0,1)\big).
\end{gathered}
\end{eqnarray}
Since $H^2_E(0,\ell)$ is a linear subspace of $L^2(0,\ell)$, we obtain, by unicity of limit, that $v = f_1 \in H^2_E(0,\ell)$.
\newline
Furthermore, from the convergence of $(u_n)_{n\in\mathbb{N}}$ to $u$ in $L^2\big((0,\ell)\times(0,1)\big)$, there exists a subsequence $(u_{n_k})_{k\in\mathbb{N}}$ such that:
\begin{gather} \label{eq:13}
u_{n_{k}} \underset{k \to + \infty}{\longrightarrow} u \; \text{ a.e. in } (0,\ell) \times (0,1),
\\
\big| u_{n_{k}} \big| \leqslant h \; \text{ a.e. in } (0,\ell) \times (0,1),
\end{gather}
with some function $h \in L^2\big((0,\ell) \times (0,1)\big)$. Then, we get:
\begin{gather} \label{eq:14}
u_{n_{k}} \varphi_s \underset{k \to + \infty}{\longrightarrow} u \varphi_s  \; \text{ a.e. in } (0,\ell) \times (0,1),
\\
\big| u_{n_{k}} \varphi_s \big| \leqslant \Big( \underset{(x,s) \in (0,\ell)\times(0,1)}{\sup} \big| \varphi_s(x,s) \big| \Big) h \; \text{ a.e. in } (0,\ell) \times (0,1).
\end{gather}
Applying Lebesgue's dominated convergence Theorem, we obtain:
\begin{align} 
\int_{0}^{\ell} \int_{0}^{1} (u_{n_{k}})_s(x,s) \varphi(x,s) \, ds \, dx  & = -\int_{0}^{\ell} \int_{0}^{1} u_{n_{k}}(x,s) \varphi_s(x,s) \, ds \, dx \nonumber
\\
\begin{split} \label{eq:15}
\int_{0}^{\ell} \int_{0}^{1} (u_{n_{k}})_s(x,s) \varphi(x,s) \, ds \, dx  & \underset{k \to + \infty}{\longrightarrow} - \int_{0}^{\ell} \int_{0}^{1} u(x,s) \varphi_s(x,s) \, ds \, dx 
\\
& \qquad \qquad = \int_{0}^{\ell} \int_{0}^{1} u_s(x,s) \varphi(x,s) \, ds \, dx,
\end{split}
\end{align}
for all $\varphi \in \mathcal{D}((0,\ell) \times (0,1))$. Consequently, $(u_{n_{k}})_s \to u_s$ a.e. in $(0,\ell) \times (0,1)$. By the density of $\mathcal{D}\big((0,\ell) \times (0,1)\big)$ in $L^2\big((0,\ell) \times (0,1)\big)$ and the uniqueness of the limit, we readily deduce that $-\tau^{-1} (u_{n_{k}})_s \to -\tau^{-1} u_s = f_3$ in $L^2\big((0,\ell) \times (0,1)\big)$, which implies that $u \in L^2\big(0,\ell; H^1(0,1)\big)$.
In addition, since $y_n \to y$ in $H^2_E(0,\ell)$, it follows that $\displaystyle y_n'' \underset{n \to +\infty}{\longrightarrow} y'', \, y_n'\underset{n \to +\infty}{\longrightarrow} y'$ and $\displaystyle y_n \underset{n \to +\infty}{\longrightarrow} y$ in $L^2(0,\ell)$. Consequently, by a similar subsequence extraction argument, there exists a subsequence  $(y_{n_k})_{k\in\mathbb{N}}$ such that:
\begin{align} \label{eq:19}
\prec y_{n_{k}}'' - y'', \varphi'' \succ_{L^2(0,\ell)} - T \prec y_{n_{k}} - y, \varphi'' \succ_{L^2(0,\ell)} \underset{k \to +\infty}{\to} 0, \nonumber \\
\prec \Big( y_{n_{k}}^{(4)} - T y_{n_{k}}'' \Big) - \Big( y^{(4)} - T y'' \Big), \varphi \succ_{L^2(0,\ell)} \underset{k \to +\infty}{\longrightarrow} 0, 
\end{align}
for all $\varphi \in \mathcal{D}(0,\ell)$, after a double integration by parts. The density of $\mathcal{D}(0,\ell)$ in $L^2(0,\ell)$ implies that $y_{n_{k}}^{(4)} - T y_{n_{k}}'' \underset{k \to +\infty}{\longrightarrow} y^{(4)} - T y''$ in $L^2(0,\ell)$. Hence, $y \in H^4(0,\ell)$. 
To verify that $y$ satisfies the boundary conditions, we identify $y$ with its continuous representative in $\mathcal{C}^3[0,\ell]$ via the Sobolev embedding theorem, and apply the sequential characterization of limits. Thus, $Y \in \mathscr{D}(\mathbb{A})$. 
\newline
Furthermore, for almost every $x \in (0,\ell)$, the map $s \mapsto u(x,s)$ belongs to $H^1(0,1)$. Since $H^1(0,1)$ is continuously embedded in $\mathcal{C}[0,1]$, identifying $u$ with its continuous representative with respect to $s$ immediately yields $u_n(\cdot,1) \to u(\cdot,1)$ as $n \to +\infty$.
Hence, $\displaystyle -y_{n_{k}}^{(4)} + T y_{n_{k}}'' - \alpha u_{n_{k}}(\cdot,1)$ converges to $ f_2 = -y^{(4)} + T y'' - \alpha u(\cdot,1)$ in $L^2(0,\ell)$ as  $k \to +\infty$. It follows immediately that $F = \mathbb{A} Y$, thereby completing the proof.
\end{proof}

\begin{lemma} \label{l2}
The linear operator $\mathbb{A}$, defined by \eqref{eq:6}--\eqref{eq:6a}, is invertible and its inverse $\mathbb{A}^{-1}$ is compact on $\mathscr{H}$.
\end{lemma}

\begin{proof}
First, we prove that the linear operator $\mathbb{A}$ is invertible. To this end, let $F = (f, g, h) \in \mathscr{H}$ and consider the equation $\mathbb{A}Y = F$ for $Y = (y, v, u) \in \mathscr{D}(\mathbb{A})$. This equation is equivalent to the following system:
\begin{equation} \label{eq:20}
	\left\{
	\begin{array}{>{\displaystyle}l}
	v = f,  \\
	-y^{(4)} + Ty'' - \alpha u(\cdot,1) = g, \\
	- \tau^{-1} u_s(\cdot, s) = h(\cdot,s).
	\end{array}
	\right.
\end{equation}
The first equation in the above system immediately yields $v = f$, which allows us to integrate the equation $\eqref{eq:20}_3$ explicitly as:
\begin{align} \label{eq:21}
u(\cdot,s) = - \tau \int_{0}^{s} h(\cdot,r) \, dr + f,
\end{align}
in $(0,1)$. Inserting this expression in the equation $\eqref{eq:20}_2$, we obtain:
\begin{align}  \label{eq:22}
-y^{(4)} + Ty'' = g + \alpha \left( f - \tau \int_{0}^{1} h(\cdot,r) \, dr \right) \nonumber
\\
y^{(4)} - Ty'' + \psi_\alpha(f,g,h) = 0,
\end{align}
where $\psi_\alpha(f,g,h) := g + \alpha \left( f - \tau \int_{0}^{1} h(\cdot,r) \, dr \right)$. The characteristic equation associated with the homogeneous part of \eqref{eq:22} reads:
\begin{gather}  \label{eq:23}
r^4 - T r^2 = 0
\end{gather}
and its roots are given by
\begin{gather}  \label{eq:25}
r = 0, \; r = \varrho \text{ and } r = -\varrho, 
\end{gather}
where $\varrho = \sqrt{T}$. Consequently, the homogeneous solution $y_H := y_H(x)$ is given by:
\begin{gather}  \label{eq:26}
y_H = c_1 + c_2 x + c_3 e^{\varrho x} + c_4 e^{-\varrho x}, \quad c_i \in \mathbb{R}, \, i \in \{1, \, 2, \, 3, \, 4\}.
\end{gather}
Let $y_P$ be a particular solution of the non-homogeneous equation \eqref{eq:22} associated with the source term $\psi_\alpha$. It can be found by varying the constants $c_i$. 
Substituting $y$ by $y_P$ in non-homogeneous equation \eqref{eq:22} yields:
\begin{gather}  \label{eq:27}
\underbrace{
\begin{pmatrix}
1 & x & e^{\varrho x} & e^{-\varrho x} \\
0 & 1 & \varrho e^{\varrho x} & -\varrho e^{-\varrho x} \\
0 & 0 & \varrho^2 e^{\varrho x} & \varrho^2 e^{-\varrho x} \\
0 & 0 & \varrho^3 e^{\varrho x} & -\varrho^3 e^{-\varrho x}
\end{pmatrix}}_{
W(x)
}
\underbrace{
\begin{pmatrix}
c_1'(x) \\ c_2'(x) \\ c_3'(x) \\ c_4'(x)
\end{pmatrix}}_{
C'(x)
}
=
\begin{pmatrix}
0 \\ 0 \\ 0 \\ -\psi_\alpha(f,g,h)
\end{pmatrix}
\end{gather}
Let $W(x)$ be the matrix associated with the above fundamental system.
The Wronskian $\Delta := \det(W(x))$ is given by:
\begin{align}  \label{eq:28}
\Delta 
= \begin{vmatrix}
1 & x & e^{\varrho x} & e^{-\varrho x} \\
0 & 1 & \varrho e^{\varrho x} & -\varrho e^{-\varrho x} \\
0 & 0 & \varrho^2 e^{\varrho x} & \varrho^2 e^{-\varrho x} \\
0 & 0 & \varrho^3 e^{\varrho x} & -\varrho^3 e^{-\varrho x}
\end{vmatrix} 
= \begin{vmatrix}
\varrho^2 e^{\varrho x} & \varrho^2 e^{-\varrho x} \\
\varrho^3 e^{\varrho x} & -\varrho^3 e^{-\varrho x}
\end{vmatrix} 
= - 2\varrho^5.
\end{align}
By Cramer's rule, we obtain $ c_i'(x) := \dfrac{\Delta_i}{\Delta}, \; i=1, \, 2, \, 3, \, 4$ , where:
\begin{align}  \label{eq:29}
\Delta_1 
& = \begin{vmatrix}
0 & x & e^{\varrho x} & e^{-\varrho x} \\
0 & 1 & \varrho e^{\varrho x} & -\varrho e^{-\varrho x} \\
0 & 0 & \varrho^2 e^{\varrho x} & \varrho^2 e^{-\varrho x} \\
-\psi_\alpha(f,g,h) & 0 & \varrho^3 e^{\varrho x} & -\varrho^3 e^{-\varrho x}
\end{vmatrix} 
= 2 x \varrho^3 \psi_\alpha(f,g,h)
\end{align}
\begin{align}  \label{eq:30}
\Delta_2 
& = \begin{vmatrix}
1 & 0 & e^{\varrho x} & e^{-\varrho x} \\
0 & 0 & \varrho e^{\varrho x} & -\varrho e^{-\varrho x} \\
0 & 0 & \varrho^2 e^{\varrho x} & \varrho^2 e^{-\varrho x} \\
0 & -\psi_\alpha(f,g,h) & \varrho^3 e^{\varrho x} & -\varrho^3 e^{-\varrho x}
\end{vmatrix}
= -2 \varrho^3 \psi_\alpha(f,g,h) 
\end{align}
\begin{align}  \label{eq:31}
\Delta_3
& = \begin{vmatrix}
1 & x & 0 & e^{-\varrho x} \\
0 & 1 & 0 & -\varrho e^{-\varrho x} \\
0 & 0 & 0 & \varrho^2 e^{-\varrho x} \\
0 & 0 & -\psi_\alpha(f,g,h)  & -\varrho^3 e^{-\varrho x}
\end{vmatrix}
= \varrho^2 e^{-\varrho x} \psi_\alpha(f,g,h)
\end{align}
\begin{align}  \label{eq:32}
\Delta_4 
& = \begin{vmatrix}
1 & x & e^{\varrho x} & 0 \\
0 & 1 & \varrho e^{\varrho x} & 0 \\
0 & 0 & \varrho^2 e^{\varrho x} & 0 \\
0 & 0 & \varrho^3 e^{\varrho x} & -\psi_\alpha(f,g,h)
\end{vmatrix}
= -\varrho^2 e^{\varrho x} \psi_\alpha(f,g,h).
\end{align}
Specifically, we find:
\begin{equation}  \label{eq:33}
\begin{aligned}
& c_1'(x) = -\frac{x \psi_\alpha(f,g,h)}{\varrho^2} \qquad  c_2'(x) = \frac{\psi_\alpha(f,g,h)}{\varrho^2} \\ 
& c_3'(x) = -\frac{e^{-\varrho x} \psi_\alpha(f,g,h)}{2 \varrho^3} \qquad c_4'(x) = \frac{e^{\varrho x} \psi_\alpha(f,g,h)}{2 \varrho^3},
\end{aligned}
\end{equation}
and consequently, a particular solution is given by,
\begin{equation} \label{eq:34}
\begin{aligned}
y_P & = \frac{1}{\varrho^2} \int_{0}^{x} (x-r) \psi_\alpha(f,g,h)(r) \, dr \\
& \quad - \frac{1}{\varrho^3} \int_{0}^{x} \sinh \Big( \varrho(x-r) \Big)\psi_\alpha(f,g,h)(r) \, dr.
\end{aligned}
\end{equation}
Hence, the general solution to the non-homogeneous equation \eqref{eq:22} is given by:
\begin{equation}  \label{eq:35}
\begin{gathered}
y(x) = c_1 + c_2 x + c_3 e^{\varrho x} + c_4 e^{-\varrho x} + 
\frac{1}{\varrho^2} \int_{0}^{x} (x-r) \psi_\alpha(f,g,h)(r) \, dr \\
- \frac{1}{\varrho^3} \int_{0}^{x} \sinh \Big( \varrho(x-r) \Big)\psi_\alpha(f,g,h)(r) \, dr.
\end{gathered}
\end{equation}
To compute the coefficients $c_i, \, i=1, \, 2, \, 3, \, 4$ in \eqref{eq:35}, we apply the boundary conditions at $x = 0$ and $x = \ell$:
\begin{equation}  \label{eq:36}
	\left\{
	\begin{array}{>{\displaystyle}l}
	y(0) = y'(0) = 0,  \\
	y''(\ell) = 0, \\
	y^{(3)}(\ell) - T y'(\ell) = \kappa f(\ell).
	\end{array}
	\right.
\end{equation}
The successive derivatives of the solution $y$ are the following expressions:
\begin{gather}  \label{eq:38}
\begin{aligned}
y'(x) &= c_2 + \varrho \big( c_3 e^{\varrho x} - c_4 e^{-\varrho x} \big) \\
& \quad + \, \frac{1}{\varrho^2} \int_{0}^{x} \Big( 1 - \cosh\big( \varrho(x-r) \big) \Big) \psi_\alpha(f,g,h)(r) \, dr 
\end{aligned} 
\\
y''(x) = \varrho^2 \big( c_3 e^{\varrho x} + c_4 e^{-\varrho x} \big) - \frac{1}{\varrho} \int_{0}^{x} \sinh\big( \varrho(x-r) \big) \psi_\alpha(f,g,h)(r) \, dr \\
y^{(3)}(x) = \varrho^3 \big( c_3 e^{\varrho x} - c_4 e^{-\varrho x} \big) - \frac{1}{\varrho} \int_{0}^{x} \cosh\big( \varrho(x-r) \big) \psi_\alpha(f,g,h)(r) \, dr.
\end{gather}
The boundary condition $\eqref{eq:36}_1$ is equivalent to the following algebraic system:
\begin{equation}   \label{eq:39}
	\left\{
	\begin{array}{>{\displaystyle}l}
	c_1 + c_3 + c_4 = 0 \\
	c_2 + \varrho \big( c_3 - c_4 \big) = 0.
	\end{array}
	\right.
\end{equation}
The condition $\eqref{eq:36}_2$ becomes:
\begin{gather}  \label{eq:40}
c_3 e^{\varrho \ell} + c_4 e^{-\varrho \ell} = \frac{1}{\varrho^3} \int_{0}^{\ell} \sinh\Big( \varrho(\ell-r) \Big) \psi_\alpha(f,g,h)(r) \, dr.
\end{gather}
The third boundary condition $\eqref{eq:36}_3$ rewrites as:
\begin{equation}
\begin{gathered}
- T c_2 + \Big( \varrho^3 - T \varrho \Big) c_3 e^{\varrho \ell} 
- \Big( \varrho^3 - T \varrho \Big) c_4 e^{-\varrho \ell} 
= \kappa f(\ell) 
\\
+ \int_{0}^{\ell} \left( -T \frac{\cosh\big( \varrho(\ell-r) \big) - 1}{\varrho^2} + \frac{\cosh\big( \varrho(\ell-r) \big)}{\varrho}  \right) \psi_\alpha(f,g,h)(r) \, dr. 
\end{gathered} \label{eq:41}
\end{equation}
Ultimately, the coefficients $c_i, \, i=1, 2, 3, 4$ are found to satisfy the algebraic system:
\begin{equation}  \label{eq:43}
	\left\{
	\begin{array}{>{\displaystyle}l}
	c_1 + c_3 + c_4 = 0 \\
	c_2 + \varrho \Big( c_3 - c_4 \Big) = 0 \\
	e^{\varrho \ell} c_3 + e^{-\varrho \ell} c_4  = \varphi_1(\ell) \\
	-T c_2 + \pi(\varrho) e^{\varrho \ell} c_3 - \pi(\varrho) e^{-\varrho \ell} c_4 = \varphi_2 (\ell),
	\end{array}
	\right.
\end{equation}
where the function $\pi(\cdot)$ and constants $\varphi_1(\ell), \, \varphi_2(\ell)$ are given by:
\begin{equation}  \label{eq:44}
\begin{aligned}
& \pi(\varrho) := \varrho^3 - T \varrho \qquad \varphi_1(\ell) := \frac{1}{\varrho^3} \int_{0}^{\ell} \sinh\Big( \varrho(\ell-r) \Big) \psi_\alpha(f,g,h)(r) \, dr, \\
& \varphi_2(\ell) :=
\frac{1}{\varrho^2} \int_{0}^{\ell} \left(  \frac{\cosh\Big( \varrho(\ell-r) \Big) - 1}{\varrho^2} 
+ \frac{\cosh\Big( \varrho(\ell-r) \Big)}{\varrho}  
\right) \psi_\alpha(f,g,h)(r) \, dr \\
& \qquad \qquad + \kappa f(\ell).
\end{aligned}
\end{equation}
A direct computation of the determinant $\Delta_c$ associated with the linear system \eqref{eq:43} yields $\Delta_c = -2 \pi(\varrho) + 2 P \varrho \cosh(\varrho \ell) \neq 0$, which guarantees the existence and uniqueness of the coefficients $c_i$. It follows that the linear operator $\mathbb{A}$ is invertible, and its inverse $\mathbb{A}^{-1}: \mathscr{H} \to \mathscr{D}(\mathbb{A})$ is continuous by the closed graph Theorem.

We now prove the compactness of $\mathbb{A}^{-1}$. We recall that $\mathscr{D}(\mathbb{A}) \subset H^4(0,1) \times H^2(0,1) \times L^2\big(0,\ell; H^1(0,1) \big))$. 
Since the domain $\mathscr{D}(\mathbb{A})$ is compactly embedded into $\mathscr{H}$ by the Rellich-Kondrachov theorem, and the resolvent $\mathbb{A}^{-1}: \mathscr{H} \to \mathscr{D}(\mathbb{A})$ is bounded, it follows immediately that $\mathbb{A}^{-1}$ is a compact operator on $\mathscr{H}$.
\end{proof}

We, now, state the first main result of this paper.
\begin{theorem} \label{th1}
The linear operator $\mathbb{A}$ generates a $C_0$-semigroup on $\mathscr{H}$. Consequently, for any initial datum $Y_0 \in \mathscr{H}$, system \eqref{eq:p1} admits a unique mild solution $Y \in C([0, +\infty), \mathscr{H})$. Furthermore, if $Y_0 \in \mathscr{D}(\mathbb{A})$, then $Y$ is a classical solution satisfying $Y \in C([0, +\infty), \mathscr{D}(\mathbb{A})) \cap C^1([0, +\infty), \mathscr{H})$.
\end{theorem}

\begin{proof}
First, we prove the dissipativity of the linear operator $\mathbb{A} - m_\alpha I$. We have:
\begin{align} 
\prec \mathbb{A}Y, Y \succ &= \prec 
\begin{pmatrix}
v \\ -y^{(4)} + Ty'' - \alpha u(\cdot, 1) \\ -\tau^{-1} u_s(\cdot, s)
\end{pmatrix},
\begin{pmatrix}
y \\ v \\ u
\end{pmatrix}
\succ  \nonumber
\\
\begin{split} \label{eq:45}
\prec \mathbb{A}Y, Y \succ &= \int_0^\ell \Big( -y^{(4)}(x) + T y''(x) - \alpha u(x, 1) \Big) v(x) \, dx \\
& \quad + \int_{0}^{\ell} v''(x) y''(x) \, dx + T \int_{0}^{\ell} v'(x) y'(x) \, dx \\
& \quad - \xi \tau^{-1} \int_{0}^{\ell} \int_{0}^{1} u_s(x,s) u(x,s) \, ds \, dx.
\end{split}
\end{align}
Performing successive integrations by parts and invoking the boundary conditions $\eqref{eq:p2}_3$ and $\eqref{eq:p2}_5$, we obtain:
\begin{gather} \label{eq:47}
\begin{aligned}
\int_{0}^{\ell} y^{(4)}(x) v(x) \, dx 
& = \Big[ y^{(3)}(x)v(x) - y''(x)v'(x) \Big]_0^\ell + \int_{0}^{\ell} y''(x) v''(x) \, dx 
\\
& = y^{(3)}(\ell)v(\ell) + \int_{0}^{\ell} y''(x) v''(x) \, dx 
\end{aligned}
\\ 
\begin{aligned}
\int_{0}^{\ell} y''(x) v(x) \, dx 
& = \Big[ y'(x)v(x) \Big]_0^\ell - \int_{0}^{\ell} y'(x) v'(x) \, dx 
\\
& = y'(\ell)v(\ell) - \int_{0}^{\ell} y'(x) v'(x) \, dx 
\end{aligned} \label{eq:47a}
\\ 
\begin{aligned}
\int_{0}^{1} u_s(\cdot,s)u(\cdot,s) \, ds 
& =
\frac{1}{2} \int_{0}^{1} \frac{d}{ds} \big( u^2(\cdot,s) \big) \, ds
= 
\Big[ u^2(\cdot,s) \Big]_{s=0}^{s=1} 
\\
& = \frac{1}{2} \Big( u^2(\cdot, 1) - u^2(\cdot, 0) \Big).
\end{aligned}\label{eq:47b}
\end{gather}
By combining the expressions \eqref{eq:47}--\eqref{eq:47b} and using the boundary condition $\eqref{eq:p2}_6$, it follows that:
\begin{align} \label{eq:48}
\begin{split}
\prec \mathbb{A}Y, Y \succ 
&= \Big( Ty'(\ell) - y^{(3)}(\ell) \Big) v(\ell) - \frac{\xi \tau^{-1}}{2} \int_{0}^{\ell} \Big( u^2(x,1) - u^2(x,0) \Big) \, dx
\\
& \quad 
- \alpha \int_{0}^{\ell} u(x,1)v(x) \, dx
\\
&= - \frac{\xi \tau^{-1}}{2} \int_{0}^{\ell} \Big( u^2(x,1) - v^2(x) \Big) \, dx 
- \alpha \int_{0}^{\ell} u(x,1)v(x) \, dx
\\
& \quad 
-\kappa v^2(\ell).
\end{split} 
\end{align}  
By applying Young's inequality to the cross term $-\alpha u(x,1)v(x)$ for all $x \in (0,\ell)$, we obtain:
\begin{align}  \label{eq:49}
-\alpha u(x,1)v(x) \leqslant \frac{\varepsilon_1}{2} u^2(x,1) + \frac{\alpha^2}{2\varepsilon_1} v^2(x), \quad \forall  \, \varepsilon_1>0.
\end{align}
Consequently, the identity \eqref{eq:48} leads to the following estimates:
\begin{align}  \label{eq:50}
\prec \mathbb{A}Y, Y \succ 
& \leqslant 
\frac{\alpha^2 \varepsilon_1^{-1} + \xi \tau^{-1}}{2} \int_{0}^{\ell} v^2(x) \, dx \nonumber
+ \frac{\varepsilon_1 - \xi \tau^{-1}}{2} \int_{0}^{\ell} u^2(x,1) \, dx -\kappa v^2(\ell) 
\\
\prec \mathbb{A}Y, Y \succ 
& \leqslant 
\Big( \frac{\xi}{2\tau} + \frac{\tau \alpha^2}{\xi} \Big) \int_{0}^{\ell} v^2(x) \, dx
- \frac{\xi \tau^{-1}}{4} \int_{0}^{\ell} u^2(x,1) \, dx 
-\kappa v^2(\ell),
\end{align}
after choosing $\varepsilon_1 = \frac{\xi \tau^{-1}}{2}$. Hence, we arrive at the final estimate:
\begin{align} \label{eq:51}
\prec \Bigg( \mathbb{A} - \bigg( \underbrace{\frac{\xi}{2\tau} + \frac{\tau \alpha^2}{\xi} \bigg)}_{m_\alpha} \Bigg)Y, Y \succ 
& \leqslant - \frac{\xi \tau^{-1}}{4} \int_{0}^{1} u^2(x,1) \, dx
-\kappa v^2(\ell) 
\leqslant 0,
\end{align}
where $m_\alpha = \frac{\xi}{2\tau} + \frac{\tau \alpha^2}{\xi} > 0$. Thus, $\mathbb{A} - m_\alpha I$ is dissipative on $\mathscr{H}$.

To complete the proof that the operator $\mathbb{A} - m_\alpha I$ generates a $C_0$-semigroup via the L\"umer-Phillips Theorem, it remains to establish its maximality. Depending on the values of the coefficient $\alpha$, we distinguish the following two cases:

\paragraph{Case 1:} $\alpha \geqslant 0$.
We prove that the operator $\lambda I - \mathbb{A}$ is surjective for any $\lambda>0$. This consists in showing that for every  $F = (f, g, h) \in \mathscr{H}$, the resolvent equation $(\lambda I - \mathbb{A})Y = F$ admits a solution in $\mathscr{D}(\mathbb{A})$. Explicitly, this equation reads:
\begin{equation} \label{eq:52}
	\left\{
	\begin{array}{>{\displaystyle}l}
	\lambda y - v = f,  \\
	\lambda v + y^{(4)} - Ty'' + \alpha u(\cdot,1) = g, \\
	\lambda u(\cdot,s) + \tau^{-1} u_s(\cdot, s) = h(\cdot,s), \quad s \in (0,1).
	\end{array}
	\right.
\end{equation}
The relations $\eqref{eq:52}_1$ and $\eqref{eq:52}_3$ give:
\begin{equation} \label{eq:53}
	\left\{
	\begin{array}{>{\displaystyle}l}
	v = \lambda y - f,  \\
	u(\cdot,s) = \Big( \lambda y - f \Big) e^{-\lambda \tau s} + \tau \int_0^s h(\cdot,\omega) e^{-\lambda \tau (s-\omega)} \, d\omega \quad s \in (0,1).
	\end{array}
	\right.
\end{equation}
Substituting the expressions for $v$ and $u$ into \eqref{eq:53} yields:
\begin{eqnarray} 
& \begin{aligned}
y^{(4)} - Ty'' + \Big( \underbrace{\lambda^2 + \alpha \lambda e^{-\tau \lambda}}_{q^1_{\alpha, \tau}(\lambda)} \Big)y = g + \Big( \underbrace{\lambda + \alpha e^{-\tau \lambda}}_{q^2_{\alpha, \tau}(\lambda)} \Big)f 
\\
+ \, \big(\underbrace{- \alpha \tau e^{-\tau \lambda}}_{q^3_{\alpha, \tau}(\lambda)}\big) \int_0^1 e^{\lambda \tau \omega} h(\cdot,\omega) \, d\omega.
\end{aligned} \label{eq:55}
\end{eqnarray}
We observe that $q^i_{\alpha, \tau} > 0, \, i=1,2$ and $q^3_{\alpha, \tau} \leqslant 0$. Hence, finding the solution $Y$ reduces to solving the following system for $y$:
\begin{equation} \label{eq:56}
	\left\{
	\begin{array}{>{\displaystyle}l}
	y^{(4)} - Ty'' + q^1_{\alpha, \tau}(\lambda)y = \psi\big( f,g,h \big),\\
	y(0) = y'(0) = y''(\ell) = 0, \\
	y^{(3)}(\ell) - T y'(\ell) - \lambda \kappa y(\ell) = -\kappa f(\ell),
	\end{array}
	\right.
\end{equation}
where the function $\psi\big( f,g,h \big)$ is defined by:
\begin{gather} \label{eq:57}
\psi\big( f,g,h \big) 
= g + q^2_{\alpha, \tau}(\lambda) f + q^3_{\alpha, \tau}(\lambda) \int_0^1 e^{\lambda \tau \omega} h(\cdot,\omega) \, d\omega.
\end{gather}
For any test function $\varphi \in H^2_E(0,\ell)$, the following identity holds:
\begin{gather} \label{eq:58}
\int_{0}^{\ell} \Big( y^{(4)}(x) - Ty''(x) + q^1_{\alpha, \tau}(\lambda)y(x) \Big)\varphi(x) \, dx = \int_{0}^{\ell} \psi\big( f,g,h \big)(x) \varphi(x) \, dx.
\end{gather}
Integrating by parts and applying the boundary condition $\eqref{eq:56}_2$, we obtain:
\begin{eqnarray}
& \begin{aligned}
\int_{0}^{\ell} y^{(4)}(x) \varphi(x) \, dx 
& = 
\Big[ y^{(3)}(x) \varphi(x) - y''(x) \varphi'(x) \Big]_0^\ell + \int_{0}^{\ell} y''(x) \varphi''(x) \, dx 
\\
& = y^{(3)}(\ell) \varphi(\ell) + \int_{0}^{\ell} y''(x) \varphi''(x) \, dx 
\end{aligned} \label{eq:59}
\\
& \begin{aligned}
\int_{0}^{\ell} y''(x) \varphi(x) \, dx 
& = \Big[ y'(x) \varphi(x) \Big]_0^\ell - \int_{0}^{\ell} y'(x) \varphi'(x) \, dx 
\\
& = y'(\ell) \varphi(\ell) - \int_{0}^{\ell} y'(x) \varphi'(x) \, dx.
\end{aligned} \label{eq:59a}
\end{eqnarray}
Combining the boundary condition $\eqref{eq:56}_3$ with the identities \eqref{eq:59} and \eqref{eq:59a}, the relation \eqref{eq:58} reduces to:
\begin{align}
& \int_{0}^{\ell} y''(x) \varphi''(x) \, dx + T \int_{0}^{\ell} y'(x) \varphi'(x) \, dx + q^1_{\alpha, \tau}(\lambda) \int_{0}^{\ell} y(x) \varphi(x) \, dx \nonumber
\\
& \quad + \Big( y^{(3)}(\ell) - T y'(\ell) \Big) \varphi(\ell)= \int_{0}^{\ell} \psi\big( f,g,h \big)(x) \varphi(x) \, dx \nonumber
\\
\begin{split} \label{eq:60}
& \int_{0}^{\ell} y''(x) \varphi''(x) \, dx + T \int_{0}^{\ell} y'(x) \varphi'(x) \, dx + q^1_{\alpha, \tau}(\lambda) \int_{0}^{\ell} y(x) \varphi(x) \, dx
\\
& \quad + \lambda \kappa y(\ell) \varphi(\ell) = \int_{0}^{\ell} \psi\big( f,g,h \big)(x) \varphi(x) \, dx + \kappa f(\ell) \varphi(\ell).
\end{split}
\end{align}
Thus, the variational problem amounts to finding $y \in H^2_E(0,\ell)$ such that:
\begin{align} \label{eq:62}
\mathcal{L}(y,\varphi) = \mathcal{J}(\varphi) 
\end{align}
for all $\varphi \in H^2_E(0,\ell)$,  where the bilinear form $\mathcal{L} : H^2_E(0,\ell) \times H^2_E(0,\ell) \to \mathbb{R}$ and the linear form $\mathcal{J} : H^2_E(0,\ell) \to \mathbb{R}$ are respectively defined by:
\begin{gather} \label{eq:63}
\mathcal{L}(y,\varphi) := \int_{0}^{\ell} y'' \varphi'' \, dx + T \int_{0}^{\ell} y' \varphi' \, dx + q^1_{\alpha, \tau}(\lambda) \int_{0}^{\ell} y \varphi \, dx + \lambda \kappa y(\ell) \varphi(\ell) \\
\mathcal{J}(\varphi) := \int_{0}^{\ell} \psi\big( f,g,h \big) \varphi \, dx + \kappa f(\ell) \varphi(\ell).
\end{gather}
Furthermore, the bilinear form $\mathcal{L}$ is coercive. Indeed, for any $y \in H^2_E(0,\ell)$, we have:
\begin{equation} \label{eq:64}
\begin{aligned} 
\mathcal{L}(\varphi, \varphi) 
& \geqslant \int_{0}^{\ell} \Big( \varphi''(x)^2 + T \varphi'(x)^2 + q^1_{\alpha, \tau}(\lambda) \varphi^2(x) \Big) \, dx 
\\
\mathcal{L}(\varphi, \varphi)
& \geqslant \min \{1, \, T, \, q^1_{\alpha, \tau}(\lambda)\} \lVert \varphi \rVert^2_{H^2_E(0,1)}.
\end{aligned}
\end{equation}
Next, we establish the continuity of the applications $J$ and $L$. More precisely, we have:
\begin{align} \label{eq:65}
\Big|\mathcal{L}(y,\varphi)\Big| &\leqslant \sqrt{\int_{0}^{\ell} y''(x)^2 \, dx \int_{0}^{\ell} \varphi''(x)^2 \, dx} + T \sqrt{\int_{0}^{\ell} y'(x)^2 \, dx \int_{0}^{\ell} \varphi'(x)^2 \, dx} \nonumber
\\
& \quad + q^1_{\alpha, \tau}(\lambda) \sqrt{\int_{0}^{\ell} y(x)^2 \, dx \int_{0}^{\ell} \varphi(x)^2 \, dx} + \lambda \kappa \int_{0}^{\ell} |y'(x)| \, dx \int_{0}^{\ell} |\varphi'(x)| \, dx \nonumber
\\
& \leqslant \sqrt{\int_{0}^{\ell} y''(x)^2 \, dx \int_{0}^{\ell} \varphi''(x)^2 \, dx} + (\lambda \kappa \ell + T) \sqrt{\int_{0}^{\ell} y'(x)^2 \, dx \int_{0}^{\ell} \varphi'(x)^2 \, dx} \nonumber
\\
& \quad + q^1_{\alpha, \tau}(\lambda) \sqrt{\int_{0}^{\ell} y(x)^2 \, dx \int_{0}^{\ell} \varphi(x)^2 \, dx} \nonumber
\\
\Big|\mathcal{L}(y,\varphi)\Big| & \leqslant 3 \max \Big\{1, \lambda \kappa \ell + T, q^1_{\alpha, \tau}(\lambda) \Big\} \lVert y \rVert_{H^2_E(0,\ell)} \lVert \varphi \rVert_{H^2_E(0,\ell)}.
\end{align}
and
\begin{align} \label{eq:66}
\Big|\mathcal{J}(\varphi)\Big| &\leqslant \lVert \psi(f,g,h) \rVert_{L^2(0,\ell)} \lVert \varphi \rVert_{L^2(0,\ell)} + \kappa \int_{0}^{\ell} |f'(x)| \, dx \int_{0}^{\ell} |\varphi'(x)| \, dx \nonumber
\\
& \leqslant \lVert \psi(f,g,h) \rVert_{L^2(0,\ell)} \lVert \varphi \rVert_{L^2(0,\ell)} + \kappa \ell \lVert f' \rVert_{L^2(0,\ell)} \lVert \varphi' \rVert_{L^2(0,\ell)} \nonumber
\\
\Big|\mathcal{J}(\varphi)\Big| & \leqslant 2 \Big( \lVert \psi(f,g,h) \rVert_{L^2(0,\ell)} + \kappa \ell \lVert f' \rVert_{L^2(0,\ell)} \Big) \lVert \varphi \rVert_{H^2(0,\ell)}.
\end{align}
As for the $L^2$-norm of $\psi(f,g,h)$, it satisfies the following estimate:
\begin{align} \label{eq:67}
\lVert \psi(f,g,h) \rVert_{L^2(0,\ell)}
& \leqslant 
\lVert g \rVert_{L^2(0,\ell)} 
+ q^2_{\alpha, \tau}(\lambda) \lVert f \rVert_{L^2(0,\ell)} \nonumber
\\
& \quad + \big| q^3_{\alpha, \tau}(\lambda) \big| 
\left\lVert \int_0^1 e^{\lambda \tau \omega} h(\cdot,\omega) \, d\omega \right\rVert_{L^2(0,\ell)} \nonumber
\\
& \leqslant 
\lVert g \rVert_{L^2(0,\ell)} 
+ q^2_{\alpha, \tau}(\lambda) \lVert f \rVert_{L^2(0,\ell)} \nonumber
\\
& \quad + \big| q^3_{\alpha, \tau}(\lambda) \big| e^{\lambda \tau} 
\left\lVert \int_0^1 h(\cdot,\omega) \, d\omega \right\rVert_{L^2(0,\ell)} \nonumber
\\
& \leqslant 
\lVert g \rVert_{L^2(0,\ell)} 
+ q^2_{\alpha, \tau}(\lambda) \lVert f \rVert_{L^2(0,\ell)} 
+ \alpha \tau \left\lVert 
\int_0^1  h(\cdot,\omega) \, d\omega 
\right\rVert_{L^2(0,\ell)} \nonumber
\\
\lVert \psi(f,g,h) \rVert_{L^2(0,\ell)}
& \leqslant 
\lVert g \rVert_{L^2(0,\ell)} 
+ q^2_{\alpha, \tau}(\lambda) \lVert f \rVert_{L^2(0,\ell)} 
+ \alpha \tau \left\lVert h 
\right\rVert_{L^2( (0,\ell) \times (0,1) )}.
\end{align}
Inserting this bound into \eqref{eq:66} yields the refined estimate:
\begin{align}
\begin{split}
\Big|\mathcal{J}(\varphi)\Big| 
& \leqslant 
2 \Big( 
\lVert g \rVert_{L^2(0,\ell)} 
+ q^2_{\alpha, \tau}(\lambda) \lVert f \rVert_{L^2(0,\ell)} 
+ \kappa \ell \lVert f' \rVert_{L^2(0,\ell)} 
\\
& \qquad + \alpha \tau \left\lVert h 
\right\rVert_{L^2( (0,\ell) \times (0,1) )} 
\Big) \lVert \varphi \rVert_{H^2(0,\ell)} \nonumber
\\
\Big|\mathcal{J}(\varphi)\Big| 
& \leqslant 
2 \Big( 
2 \max \Big\{ \kappa \ell, \, q^2_{\alpha, \tau}(\lambda) \Big\}
\lVert f \rVert_{H^2(0,\ell)}
+ \lVert g \rVert_{L^2(0,\ell)} 
\\
& \qquad + \alpha \tau \left\lVert h 
\right\rVert_{L^2( (0,\ell) \times (0,1) )} 
\Big) \lVert \varphi \rVert_{H^2(0,\ell)}.
\end{split} \label{eq:68}
\end{align}
Consequently, by virtue of the Lax-Milgram theorem, we conclude that the weak formulation \eqref{eq:62} admits a unique solution $y \in H^2_E(0,\ell)$.

Conversely, assuming $y$ is a weak solution, performing successive integrations by parts yields:
\begin{equation} \label{eq:69}
\begin{gathered}
\int_{0}^{\ell} \Big( y^{(4)}(x) - T y''(x) \Big) \varphi(x) \, dx + q^1_{\alpha, \tau}(\lambda) \int_{0}^{\ell} y(x) \varphi(x) \, dx 
\\
+ \, T\Big[ y'(x)\varphi(x) \Big]_0^\ell + \lambda \kappa y(\ell) \varphi(\ell) 
- \Big[ y^{(3)}(x)\varphi(x) - y''(x)\varphi'(x) \Big]_0^\ell
\\
= \int_{0}^{\ell} \psi(f,g,h)(x) \varphi(x) \, dx + \kappa f(\ell) \varphi(\ell),
\end{gathered}
\end{equation}
for all $\varphi \in H^2_E(0,\ell)$.
Selecting $\varphi \in \mathcal{D}(0,\ell)$, we have:
\begin{align} \label{eq:70}
\int_{0}^{\ell} \Big( y^{(4)}(x) - T y''(x) + q^1_{\alpha, \tau}(\lambda) y(x) \Big) \varphi(x) \, dx = \int_{0}^{\ell} \psi(f,g,h)(x) \varphi(x) \, dx.
\end{align}
Therefore, the identity $y^{(4)} - T y'' + q^1_{\alpha, \tau}(\lambda) y = \psi(f,g,h)$ holds almost everywhere in $(0,\ell)$. Since $\mathcal{D}(0,\ell)$ is dense in $L^2(0,\ell)$, this relation extends to $L^2(0,\ell)$, which immediately implies that $y \in H^4(0,\ell)$. Consequently, we have $v = \lambda y - f \in H^2_E(0,\ell)$ and $u \in L^2(0,\ell; H^1(0,1))$.
Coming back to $H^2_E(0,\ell)$, we obtain:
\begin{gather} \label{eq:71}
- \Big[ y^{(3)}(x)\varphi(x) - y''(x)\varphi'(x) \Big]_0^\ell + T \Big[ y'(x)\varphi(x) \Big]_0^\ell + \kappa \Big(\lambda y(\ell) - f(\ell)\Big) \varphi(\ell) = 0 \nonumber
\\
\Big( -y^{(3)}(\ell) + T y'(\ell) + \kappa \big(\lambda y(\ell) - f(\ell)\big) \Big) \varphi(\ell) + y''(\ell) \varphi'(\ell) = 0.
\end{gather}
As $\varphi$ is arbitrary, the relations $\eqref{eq:56}_2$ and $\eqref{eq:56}_3$ are satisfied. Then $y$ solves \eqref{eq:56}.  Hence, $y$ solves \eqref{eq:56}, which proves that the operator $\lambda I - \mathbb{A}$ is surjective for all $\lambda > 0$. Since $\lambda + m_\alpha > 0$, it follows that $(\lambda + m_\alpha) I - \mathbb{A} = \lambda I - (\mathbb{A} - m_\alpha I)$ is also surjective. Consequently, the L\"umer--Phillips theorem ensures that $\mathbb{A} - m_\alpha I$ generates a contractive $C_0-$semigroup on $\mathscr{H}$.  Finally, by invoking the bounded perturbation theorem (see \cite[Theorem 1.1 p.76]{pazy1983}), we conclude that $\mathbb{A}$ generates a strongly continuous semigroup on $\mathscr{H}$.

\paragraph{Case 2:} $\alpha < 0$.

Since $\mathbb{A} - m_\alpha I$ is dissipative, then $\lambda I - \mathbb{A}$ is  dissipative for all $\lambda > m_\alpha$. It follows that, $\lambda I - \mathbb{A}$ is injective for all $\lambda > m_\alpha$. 
In addition, by virtue of Lemmas \ref{l1} and \ref{l2}, the following identity holds:
\begin{gather} \label{eq:72}
\lambda I - \mathbb{A} = - \mathbb{A} \big( I - \lambda \mathbb{A}^{-1} \big), \quad \forall \lambda > m_\alpha.
\end{gather}
Due to the compactness of linear operator $\lambda \mathbb{A}^{-1}$,  the Fredholm alternative theorem ensures that $I - \lambda \mathbb{A}^{-1}$ is surjective. This, in turn, implies the surjectivity of $\lambda I - \mathbb{A}$ as a composition of surjective operators. By virtue of the L\"umer--Phillips theorem, we deduce that $\mathbb{A} - m_\alpha I$ generates a contraction $C_0-$semigroup on $\mathscr{H}$. Finally, the bounded perturbation Theorem implies that $\mathbb{A}$ generates a $C_0-$semigroup.
\end{proof}

\section{Exponential stability of the system: some sufficient geometric conditions}
This section is devoted to the exponential stabilization of system \eqref{eq:p1}. By constructing an appropriate Lyapunov functional, we identify a region of exponential stability for the system. This region is characterized by specific geometric conditions arising from a system of inequalities.
 
For any regular solution $y$ of system \eqref{eq:p1}, the associated energy is defined by:
\begin{equation}
\begin{aligned}
E(t) & := \frac{1}{2} \left( \int_{0}^{\ell} y_t^2(x,t) \, dx + \int_{0}^{\ell} y_{xx}^2(x,t) \, dx + T \int_{0}^{\ell} y_x^2(x,t) \, dx \right)
\\
& \quad + \frac{\xi}{2} \int_{0}^{\ell} \int_{0}^{1} y_t^2(x,t-s\tau) \, ds \, dx, \quad t \geqslant 0,
\end{aligned} \label{eq:73}
\end{equation}
where $\xi>0$ is a constant that will be fixed later. \newline
The following technical lemmas provide the necessary estimates for the time derivatives of the energy and additional auxiliary functionals to be introduced.

\begin{lemma} \label{l3}
Let $E(t)$ be the energy functional defined by \eqref{eq:73}. For any solution of system \eqref{eq:p1}, the time derivative of the energy satisfies the following estimate for all $t \geqslant 0$:
\begin{equation} \label{eq:74}
\frac{d}{dt}E(t) 
\leqslant 
-\kappa y_t^2(\ell,t)
+ \frac{|\alpha| + \xi}{2} \int_{0}^{\ell} y_t^2(x,t) , dx 
+ \frac{|\alpha| - \xi}{2} \int_{0}^{\ell} y_t^2(x, t-\tau) , dx. 
\end{equation}
\end{lemma}

\begin{proof}
We have:
\begin{align}  \label{eq:75}
\begin{split}
\frac{dE(t)}{dt}
& = 
\frac{d}{dt} \left[ 
\frac{1}{2} \left(
\int_{0}^{\ell} y_t^2(x,t) \, dx + \int_{0}^{\ell} y_{xx}^2(x,t) \, dx + T \int_{0}^{\ell} y_x^2(x,t) \, dx
\right)
\right]
\\
& \qquad \qquad 
+ \, \xi \left( 
\int_{0}^{\ell} y_t^2(x,t) \, dx - \int_{0}^{\ell} y_t^2(x,t-\tau) \, dx
\right).
\end{split}
\end{align}
First, we multiply the first equation of \eqref{eq:p1} by $y_t$, which yields:
\begin{align}  \label{eq:76}
\int_{0}^{\ell} \Big( y_{tt}(x,t) + y_{xxxx}(x,t) - T y_{xx}(x,t) + \alpha y_t(x, t-\tau) \Big) \, y_t(x,t) \, dx = 0.
\end{align}
We then rewrite the first integral term of the time derivative by isolating the time derivative operator:
\begin{align}  \label{eq:77}
\int_{0}^{\ell} y_{tt}(x,t) \, y_t(x,t) \, dx = \dfrac{d}{dt} \left( \frac{1}{2}  \int_{0}^{\ell} y_t^2(x,t) \, dx \right).
\end{align}
Next, performing integrations by parts over $(0,\ell)$ and incorporating the boundary condition $\eqref{eq:p1}_2$, we obtain:
\begin{align}  
\begin{split} \label{eq:78}
\int_{0}^{\ell} y_{xxxx}(x,t) \, y_t(x,t) \, dx 
& = \Big[ y_{xxx}(x,t) \, y_t(x,t) - y_{xx}(x,t) \, y_{tx}(x,t) \Big]_{x=0}^{x=\ell} 
\\
& \quad + \int_{0}^{\ell} y_{xx}(x,t) \, y_{txx}(x,t) \, dx 
\\
& = y_{xxx}(\ell,t) \, y_t(\ell,t) + \dfrac{d}{dt} \left( \frac{1}{2}  \int_{0}^{\ell} y_{xx}^2(x,t) \, dx \right)
\end{split}
\\
\begin{split}  \label{eq:79}
\int_{0}^{\ell} y_{xx}(x,t) \, y_t(x,t) \, dx 
& = \Big[ y_{x}(x,t) \, y_t(x,t) \Big]_{x=0}^{x=\ell} + \int_{0}^{\ell} y_{x}(x,t) \, y_{tx}(x,t) \, dx \\
& = y_{x}(\ell,t) \, y_t(\ell,t) - \dfrac{d}{dt} \left( \frac{1}{2}  \int_{0}^{\ell} y_{x}^2(x,t) \, dx \right).
\end{split}
\end{align}
From identities \eqref{eq:77}--\eqref{eq:79} and using the boundary conditions $\eqref{eq:p1}_3-\eqref{eq:p1}_4$, relation \eqref{eq:76} becomes:
\begin{align} \label{eq:80}
& \dfrac{d}{dt} \left[ \frac{1}{2} \left( \int_{0}^{\ell} y_t^2(x,t) dx + \int_{0}^{\ell} y_{xx}^2(x,t) \, dx + T \int_{0}^{\ell} y_x^2(x,t) \, dx \right) \right] \nonumber\\
& \qquad = -\alpha \int_{0}^{\ell} y_t(x, t-\tau) \, y_t(x,t) \, dx - \Big( y_{xxx}(\ell,t) - T y_x(\ell,t) \Big) y_t(\ell,t) \nonumber \\
& \qquad = -\alpha \int_{0}^{\ell} y_t(x, t-\tau) \, y_t(x,t) \, dx -\kappa y_t^2(\ell,t).
\end{align}
Consequently, identity \eqref{eq:75} becomes:
\begin{align} \label{eq:81}
\begin{split}
\frac{dE(t)}{dt}
&= - \frac{\xi}{2} \int_{0}^{\ell} \Big(  y_t^2(x,t-\tau) - y_t^2(x,t) \Big) \, dx
\\
& \qquad -\alpha \int_{0}^{\ell} y_t(x, t-\tau) \, y_t(x,t) \, dx
-\kappa y_t^2(\ell,t).
\end{split}
\end{align}
Finally, applying Young's inequality to the cross-term $-\alpha y_t(x, t-\tau) y_t(x,t)$ yields:
\begin{align} \label{eq:82}
-\alpha \int_{0}^{\ell} y_t(x, t-\tau) \, y_t(x,t) \, dx
\leqslant
\frac{|\alpha|}{2} \int_{0}^{\ell} \Big( y_t^2(x, t-\tau) + y_t^2(x,t) \Big) \, dx,
\end{align}
from which the estimate \eqref{eq:74} follows directly.
\newline
\normalshape
We can observe that the energy of system \eqref{eq:p1} is not non-increasing over time.
\end{proof}

Furthermore, we define the following two functionals associated with any solution $y$ of problem \eqref{eq:p1}:
\begin{gather} \label{eq:83}
I_1(t) := \int_0^\ell x y_x(x,t) y_t(x,t) \, dx
\\
I_2(t) := \int_0^\ell \int_{t-\tau}^{t} e^{s-t} y_t^2(x,s) \, ds \, dx, \label{eq:83a}
\end{gather}
for all $t \geqslant 0$. 

\begin{lemma} \label{l4}
Let $\varepsilon_1, \varepsilon_2 > 0$ be arbitrary constants. The time derivative of the functional $I_1(t)$ satisfies the following estimate for all $t \geqslant 0$:
\begin{equation} \label{eq:84}
\begin{aligned}
\frac{dI_1(t)}{dt} 
& \leqslant
-\frac{1}{2} \int_0^\ell y_{t}^2(x,t) \, dx 
+ \frac{-3 + T\ell^2 + \varepsilon_2 \ell}{2} \int_0^\ell y_{xx}^2(x,t) \, dx 
\\ 
& \quad 
+ \frac{\varepsilon_1 - T}{2} \int_0^\ell y_{x}^2(x,t) \, dx
+ \frac{\big( \alpha \ell \big)^2}{2\varepsilon_1} \int_{0}^{\ell}  y_t^2(x,t-\tau) \, dx
\\ 
& \quad
+ \left( \frac{\ell}{2} + \frac{\big( \kappa \ell \big)^2}{2 \varepsilon_2} \right) y_t^2(\ell,t).
\end{aligned}
\end{equation}
\end{lemma}

\begin{proof}
By definition, we have:
\begin{align} \label{eq:85}
\frac{dI_1(t)}{dt} 
& = \int_0^\ell x \Big( y_{xt}(x,t) y_t(x,t) + y_{x}(x,t) y_{tt}(x,t) \Big) \, dx.
\end{align}
Integrating the first integral term by parts over $(0,\ell)$ and using the boundary condition $\eqref{eq:p1}_2$ lead to:
\begin{align} \label{eq:86}
\int_0^\ell x y_{xt}(x,t) y_t(x,t) \, dx 
& = \frac{1}{2} \left( 
\Big[ x y_t^2(x,t) \Big]_{x=0}^{x=\ell} 
- \int_0^\ell y_{t}^2(x,t) \, dx 
\right) \nonumber
\\
& = \frac{1}{2} \left( \ell y_t^2(\ell,t) - \int_0^\ell y_{t}^2(x,t) \, dx
\right).
\end{align}
Similarly, substituting the governing equation $\eqref{eq:p1}_1$ into the second integral term of \eqref{eq:85} yields:
\begin{align} \label{eq:87}
\begin{split}
& \int_0^\ell x y_{x}(x,t) y_{tt}(x,t) \, dx 
\\
& \quad = -\int_0^\ell x y_{x}(x,t) \Big( y_{xxxx}(x,t) - T y_{xx}(x,t) + \alpha y_t(x,t-\tau)\Big) \, dx.
\end{split}
\end{align}
To evaluate this relation, we perform successive integrations by parts. First, the first integral term is rewritten as:
\begin{align} \label{eq:88}
& \int_0^\ell x y_{x}(x,t) y_{xxxx}(x,t) \, dx \nonumber
\\
& \qquad = \Big[ x y_{x}(x,t) y_{xxx}(x,t) \Big]_{x=0}^{x=\ell} 
- \int_0^\ell \Big( y_{x}(x,t) + x y_{xx}(x,t) \Big) y_{xxx}(x,t) \, dx \nonumber
\\
& \qquad = \ell y_{x}(\ell,t) y_{xxx}(\ell,t)
- \int_0^\ell \Big( y_{x}(x,t) + x y_{xx}(x,t) \Big) y_{xxx}(x,t) \, dx. 
\end{align}
In particular, observing that:
\begin{align} \label{eq:89}
\int_0^\ell y_{x}(x,t) y_{xxx}(x,t) \, dx
&= \Big[ y_{x}(x,t) y_{xx}(x,t) \Big]_{x=0}^{x=\ell}
- \int_0^\ell y_{xx}^2(x,t) \, dx \nonumber
\\
\int_0^\ell y_{x}(x,t) y_{xxx}(x,t) \, dx
&= - \int_0^\ell y_{xx}^2(x,t) \, dx,
\end{align} 
and 
\begin{align} \label{eq:90}
\int_0^\ell x y_{xx}(x,t) y_{xxx}(x,t) \, dx
&= \frac{1}{2} \left( 
\Big[ x y_{xx}^2(x,t) \Big]_{x=0}^{x=\ell}
- \int_0^\ell y_{xx}^2(x,t) \, dx 
\right) \nonumber
\\
\int_0^\ell x y_{xx}(x,t) y_{xxx}(x,t) \, dx
&= -\frac{1}{2} \int_0^\ell y_{xx}^2(x,t) \, dx,
\end{align} 
after applying the boundary conditions $\eqref{eq:p1}_2$ and $\eqref{eq:p1}_3$. So, we obtain from  identities \eqref{eq:89} and \eqref{eq:90}:
\begin{align} \label{eq:91}
\int_0^\ell x y_{x}(x,t) y_{xxxx}(x,t) \, dx 
= \frac{3}{2} \int_0^\ell y_{xx}^2(x,t) \, dx 
+ \ell y_{x}(\ell,t) y_{xxx}(\ell,t).
\end{align}
Next, the following holds:
\begin{align} \label{eq:92}
\int_0^\ell x y_{x}(x,t) y_{xx}(x,t) \, dx 
& = \frac{1}{2} \left( 
\Big[ x y_{x}^2(x,t) \Big]_{x=0}^{x=\ell} 
-\int_0^\ell y_{x}^2(x,t) \, dx 
\right) \nonumber
\\
& = \frac{1}{2} \left( 
\ell y_{x}^2(\ell,t) - \int_0^\ell y_{x}^2(x,t) \, dx 
\right).
\end{align}
Therefore, exploiting relations \eqref{eq:91} and \eqref{eq:92}, identity \eqref{eq:87} reduces to:
\begin{align} 
& \int_0^\ell x y_{x}(x,t) y_{tt}(x,t) \, dx \nonumber
\\
&= - \frac{3}{2} \int_0^\ell y_{xx}^2(x,t) \, dx 
- \frac{T}{2} \int_0^\ell y_{x}^2(x,t) \, dx \nonumber \\ 
& \quad 
- \alpha \int_0^\ell x y_t(x,t-\tau) y_x(x,t) \, dx 
- \ell y_{x}(\ell,t)  y_{xxx}(\ell,t)
+ \frac{\ell T}{2} y_{x}^2(\ell,t) . \label{eq:93}
\end{align} 
Consequently, inserting relations \eqref{eq:86} and \eqref{eq:93} into \eqref{eq:85} yields:
\begin{align} 
\frac{dI_1(t)}{dt} 
& = -\frac{1}{2} \int_0^\ell y_{t}^2(x,t) \, dx 
- \frac{3}{2} \int_0^\ell y_{xx}^2(x,t) \, dx 
- \frac{T}{2} \int_0^\ell y_{x}^2(x,t) \, dx \nonumber
\\ 
& \quad 
- \alpha \int_0^\ell x y_t(x,t-\tau) y_x(x,t) \, dx + \frac{\ell}{2} y_t^2(\ell,t) + \frac{\ell T}{2} y_{x}^2(\ell,t) \nonumber
\\
& \quad
- \ell y_{x}(\ell,t)  y_{xxx}(\ell,t) \nonumber
\\
\begin{split}  \label{eq:95}
\frac{dI_1(t)}{dt} 
& = -\frac{1}{2} \int_0^\ell y_{t}^2(x,t) \, dx 
- \frac{3}{2} \int_0^\ell y_{xx}^2(x,t) \, dx 
- \frac{T}{2} \int_0^\ell y_{x}^2(x,t) \, dx
\\ 
& \quad 
- \alpha \int_0^\ell x y_t(x,t-\tau) y_x(x,t) \, dx + \frac{\ell}{2} y_t^2(\ell,t) - \frac{\ell T}{2} y_{x}^2(\ell,t) \\
& \quad - \kappa \ell y_{x}(\ell,t) y_t(\ell,t).
\end{split}
\end{align}
Nevertheless, applying Young's inequality, we have:
\begin{eqnarray}
&\begin{aligned}
&-\alpha \int_0^\ell x y_t(x,t-\tau) y_x(x,t) \, dx
\\
& \qquad \qquad \leqslant
\frac{\varepsilon_1}{2} \int_{0}^{\ell} y_x^2(x,t) \, dx
+ \frac{\alpha^2}{2\varepsilon_1} \int_{0}^{\ell} x^2 y_t^2(x,t-\tau) \, dx 
\\
& \qquad \qquad \leqslant
\frac{\varepsilon_1}{2} \int_{0}^{\ell} y_x^2(x,t) \, dx
+ \frac{\big( \alpha \ell \big)^2}{2\varepsilon_1} \int_{0}^{\ell}  y_t^2(x,t-\tau) \, dx,
\end{aligned} \label{eq:97}
\\
&\begin{aligned}
- \kappa \ell y_{x}(\ell,t)  y_{t}(\ell,t)
& \leqslant
\frac{\varepsilon_2}{2} y_{x}^2(\ell,t)
+ \frac{\big( \kappa \ell \big)^2}{2 \varepsilon_2} y_{t}^2(\ell,t),
\end{aligned} \label{eq:98}
\end{eqnarray}
for all $\varepsilon_1, \, \varepsilon_2>0$. Since, we have after application of Cauchy-Schwarz inequality:
\begin{align}  \label{eq:99}
y_x^2(\ell,t) 
=
\left( \int_0^\ell y_{xx}(x,t) \, dx \right)^2
\leqslant
\ell \int_0^\ell y_{xx}^2(x,t) \, dx,
\end{align}
then the inequality \eqref{eq:98} becomes:
\begin{align}  \label{eq:100}
- \kappa \ell y_{x}(\ell,t)  y_{t}(\ell,t)
\leqslant
\frac{\varepsilon_2}{2} \ell \int_0^\ell y_{xx}^2(x,t) \, dx
+ \frac{\big( \kappa \ell \big)^2}{2 \varepsilon_2} y_{t}^2(\ell,t).
\end{align}
Finally, incorporating estimates \eqref{eq:97} and \eqref{eq:100} into \eqref{eq:95} leads to the desired result.
\end{proof}

\begin{lemma} \label{l5}
The time derivative of the functional $I_2(t)$, defined by \eqref{eq:83a}, satisfies the following estimate:
\begin{equation} \label{eq:101}
\begin{split}
\frac{d I_2(t)}{dt} 
& \leqslant \int_{0}^{\ell} y_t^2(x,t) \, dx - e^{-\tau}  \int_{0}^{\ell} \int_{t-\tau}^{t} y_t^2(x,s) \, ds \, dx
\\
& \quad - e^{-\tau} \int_{0}^{\ell} y_t^2(x,t-\tau) \, dx.
\end{split}
\end{equation}
\end{lemma}

\begin{proof}
We refer the reader to \cite{ammari2010, feng2021} for a detailed proof of this result.
\end{proof}


Finally, we introduce  the Lyapunov functional $\mathscr{V}(t)$ defined by:
\begin{align}  \label{eq:104}
\mathscr{V}(t) := E(t) + \sum_{i=1}^{2} \delta_i I_i(t),
\end{align}
where $\delta_1, \delta_2 > 0$ are constants, and the functionals $E(t)$, $I_1(t)$, and $I_2(t)$ are given in \eqref{eq:73}, \eqref{eq:83}, and \eqref{eq:83a}.
To establish the exponential stability of system \eqref{eq:p1}, we first prove the equivalence between the energy $E(t)$ and the Lyapunov functional $\mathscr{V}(t)$:

\begin{proposition} \label{p}
For a sufficiently small $\delta_1 > 0$, the functional $\mathscr{V}(t)$ is equivalent to the energy $E(t)$. Specifically, there exist two positive constants $\gamma_1$ and $\gamma_2$ such that, for all $t \geqslant 0$,
\begin{equation} \label{eq:105}
\gamma_1 E(t) \leqslant \mathscr{V}(t) \leqslant \gamma_2 E(t).
\end{equation}
\end{proposition}

\begin{proof}
First, we estimate $|I_1(t)|$. Applying $\varepsilon$-Young's inequality on the integrand $x y_x y_t$ yields, for any $\varepsilon_3 > 0$:
\begin{align}  \label{eq:106}
|I_1(t)| 
& \leqslant
\frac{\varepsilon_3}{2} \int_{0}^{\ell} y_t^2(x,t) \, dx
+ \frac{\ell^2}{2\varepsilon_3} \int_{0}^{\ell} y_x^2(x,t) \, dx
 \nonumber \\
|I_1(t)|
& \leqslant \max \Big\{\varepsilon_3; \, \frac{\ell^2}{\varepsilon_3 T}\Big\} E(t).
\end{align}
Setting $\varepsilon_3 := \ell$, we obtain:
\begin{align}
|I_1(t)|
& \leqslant \max \Big\{\ell; \, \frac{\ell}{T}\Big\} E(t).
\end{align}
Next, the functional $I_2(t)$ satisfies the following  estimate:
\begin{gather}  \label{eq:107}
I_2
\leqslant 
\int_0^\ell \int_{t-\tau}^{t} y_t^2(x,s) \, ds \, dx
\leqslant
\frac{2}{\xi} E(t).
\end{gather}
Consequently, by the triangle inequality, we obtain the following two-sided estimate:
\begin{gather}  \label{eq:108}
 \begingroup  \everymath{\displaystyle}
 \begin{matrix}
 - \, \delta_1 \max \Big\{\ell; \, \frac{\ell}{T}\Big\} E(t)
 \\
+ \, \delta_2 \int_0^\ell \int_{t-\tau}^{t} y_t^2(x,s) \, ds \, dx
 \end{matrix}
 \endgroup
 \leqslant
 \sum_{i=1}^{2} \delta_i I_i(t)
 \leqslant
 \begingroup
 \everymath{\displaystyle}
 \begin{matrix}
 \delta_1 \max \Big\{\ell; \, \frac{\ell}{T}\Big\} E(t)
 \\
+ \, \delta_2 \int_0^\ell \int_{t-\tau}^{t} y_t^2(x,s) \, ds \, dx
 \end{matrix}
 \endgroup \nonumber
 \\
\Bigg(
1 - \delta_1 \max \Big\{\ell; \, \frac{\ell}{T}\Big\} + \frac{2 \delta_2}{\xi} 
\Bigg) E(t)
\leqslant \mathscr{V}(t) \leqslant
\Bigg(
1 + \delta_1 \max \Big\{\ell; \, \frac{\ell}{T}\Big\} + \frac{2 \delta_2}{\xi} 
\Bigg) E(t).
\end{gather}
Hence, choosing $\delta_1$ small enough such that $\delta_1 < \Big( \max \Big\{\ell; \, \dfrac{\ell}{T}\Big\} \Big)^{-1}$, the left-hand side coefficient $\displaystyle 1 - \delta_1 \max \Big\{\ell; \, \frac{\ell}{T}\Big\}$ remains positive and the result \eqref{eq:105} is thus established, with the positive constants $\gamma_1$ and $\gamma_2$ defined by:
\begin{gather}  \label{eq:109}
\gamma_1 := 1 - \delta_1 \max \Big\{\ell; \, \frac{\ell}{T}\Big\} + \frac{2 \delta_2}{\xi} 
\qquad 
\gamma_2 := 1 + \delta_1 \max \Big\{\ell; \, \frac{\ell}{T}\Big\} + \frac{2 \delta_2}{\xi}. 
\end{gather}
\end{proof}

With all the necessary ingredients in place, we are now ready to state the second main result regarding the exponential stability of system \eqref{eq:p1}.
\begin{theorem} \label{th2}
Let $\kappa > 0$ and assume that $T \in (0, 3/\ell^2)$. There exists a positive constant $\alpha_0$ such that, for any $\alpha$ satisfying $|\alpha| < \alpha_0$, system \eqref{eq:p1} is uniformly exponentially stable. Specifically, for any solution to \eqref{eq:p1}, there exist positive constants $\zeta_1$ and $\zeta_2$, independent of the initial data, such that
\begin{equation} \label{eq:109a}
E(t) \le \zeta_1 E(0) e^{-\zeta_2 t}, \quad \forall t \ge 0.
\end{equation}
\end{theorem}

\begin{proof}
In view of Lemmas \ref{l3}--\ref{l5}, the time derivative of the Lyapunov functional satisfies:
\begin{align*}
\frac{d \mathscr{V}(t)}{dt} 
& = \frac{d E(t)}{dt} + \sum_{i=1}^{2} \delta_i \frac{d I_i(t)}{dt} 
\\
\frac{d \mathscr{V}(t)}{dt}
& \leqslant
\bigg(
- \kappa + \delta_1 \left( \frac{\ell}{2} + \frac{(\kappa \ell)^2}{2 \varepsilon_2} \right)
\bigg) y^2_t(\ell,t)
+ \Big( 
\delta_2 - \frac{1}{2} \delta_1 + \frac{|\alpha| + \xi}{2} 
\Big)\int_{0}^{\ell} y_t^2(x,t) \, dx
\\
& \quad + \frac{- 3 + T\ell^2 + \ell \varepsilon_2}{2} \delta_1 
\int_{0}^{\ell} y_{xx}^2(x,t) \, dx
+ \frac{\varepsilon_1 - T}{2} \delta_1 \int_{0}^{\ell} y_{x}^2(x,t) \, dx
\\
& \quad + \left( 
-e^{-\tau} \delta_2 + \delta_1 \frac{(\alpha \ell)^2}{2 \varepsilon_1} 
+ \frac{|\alpha| - \xi}{2} 
\right)\int_{0}^{\ell} y_t^2(x,t-\tau) \, dx
\\
& \quad - e^{-\tau} \delta_2 \int_0^\ell \int_{t-\tau}^{t} y_t^2(x,s) \, ds \, dx.
\end{align*}
First, we set $\varepsilon_1 = \frac{T}{2}$ and $\varepsilon_2 = \frac{3 - T\ell^2}{2\ell}$, with $\displaystyle T \in \left(0; \, \frac{3}{\ell^2}\right)$. 
Next, we choose $\delta_i > 0, \, i = 1, 2$ sufficiently small such that:
\begin{gather} \label{eq:110}
\delta_1 \leqslant \kappa \left( 
\frac{\ell}{2} + \frac{\kappa^2 \ell^3}{3 - T \ell^2} \right)^{-1}, \quad 
\delta_2 < \frac{\delta_1}{2}.
\end{gather}
Under these choices, we obtain:
\begin{align} \label{eq:111}
\begin{split}
\frac{d \mathscr{V}(t)}{dt} 
& \leqslant
\bigg( -\kappa + \delta_1 \left( \frac{\ell}{2} + \frac{\kappa^2 \ell^3}{3 -	T\ell^2} \right) \bigg) y_t^2(\ell,t)
\\
& \quad
 + \bigg( \delta_2 - \frac{\delta_1}{2} + \frac{|\alpha| + \xi}{2} \bigg) \int_{0}^{\ell} y_t^2(x,t) \, dx
\\
& \quad - \frac{3 - T\ell^2}{4} \int_{0}^{\ell} y_{xx}^2(x,t) \, dx
- \frac{T}{4} \delta_1 \int_{0}^{\ell} y_{x}^2(x,t) \, dx
\\
& \quad - \bigg( e^{-\tau} \delta_2 - \delta_1 \frac{(\alpha \ell)^2}{T} - \frac{|\alpha| - \xi}{2} \bigg) 
\int_{0}^{\ell} y_t^2(x,t-\tau) \, dx
\\
& \quad
 - e^{-\tau} \delta_2 \int_{0}^{\ell} \int_{t-\tau}^{t} y_t^2(x,s) \, ds \, dx.
\end{split}
\end{align} 
We denote by $\Sigma$ the $(\alpha,\xi)-$region of exponential stability associated with system \eqref{eq:p1}, defined as:
\begin{gather}  \label{eq:112}
\Sigma := \left\{
(\alpha,\xi) \in \mathbb{R} \times \mathbb{R}_+^\ast:
\begin{array}{l}
\xi < \delta_1 - 2\delta_2 - |\alpha| \\
\xi \geqslant \dfrac{2 \delta_1 \ell^2}{T} \alpha^2 + |\alpha| - 2 e^{-\tau} \delta_2 
\end{array}
\right\}.
\end{gather}
For any pair $(\alpha,\xi) \in \Sigma$, there exists a constant $\nu > 0$ such that:
\begin{gather}  \label{eq:113}
\frac{d \mathscr{V}(t)}{dt} \leqslant -\nu E(t) \leqslant -\frac{\nu}{\gamma_2} \mathscr{V}(t).
\end{gather}
Specifically, $\nu$ is given by:
\begin{gather}  \label{eq:114}
\nu := 2 \min \left\{
\frac{T}{4} \delta_1; \, 
\frac{3-T\ell^2}{4}; \,
e^{-\tau} \xi^{-1} \delta_2; \, 
\frac{\delta_1}{2}  - \delta_2 - \frac{|\alpha| + \xi}{2}
\right\}.
\end{gather}
Hence, combining \eqref{eq:113} with the equivalence property \eqref{eq:105} of Proposition \ref{p} establishes \eqref{eq:109a}.
\end{proof}

\begin{remark} ~
\begin{enumerate}[label=(\roman*)]
\item As depicted in Figure \ref{fig}, the exponential stability region $\Sigma$ is non-empty for the considered system parameters.

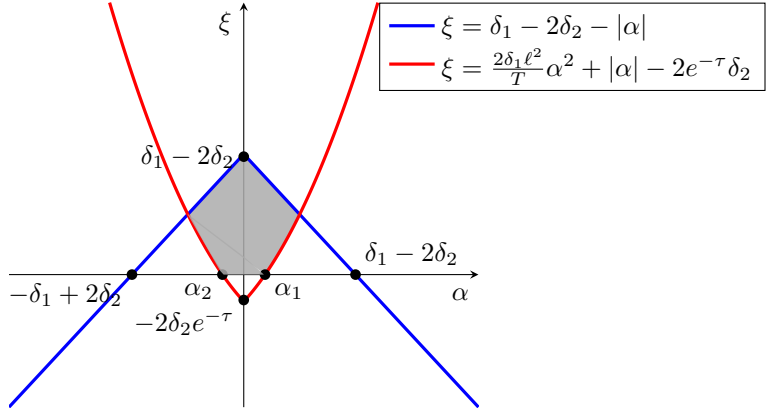
\begin{figure}[ht]
\begin{center}
\begin{tikzpicture}
\begin{axis}[
width=0.49\textwidth,
scale only axis,
axis lines = middle,
xlabel = {$\alpha$},
xlabel style={yshift=-3ex},
ylabel = {$\xi$},
ylabel style={xshift=-3ex},
xtick = \empty,
ytick= \empty,
legend style = {
at={(1.2,1)},
anchor=north,
legend cell align={left},
}
]

\path[name path=X](axis cs:-0.7,0) -- (axis cs:0.7,0);
\path[name path=Y](axis cs:0,-0.5) -- (axis cs:0,5);

\addplot[
very thick,
domain=-0.7:0.7,
samples=100,
color=blue,
name path=A
]{(-1) * abs(x) + 1/3};

\addplot[
very thick,
domain=-0.4:0.4,
samples=100,
color=red,
name path=B
]{(8/3) * x^2 + abs(x) - (1/3)*exp(-1.5)};

\fill[
name intersections={of=X and A, by = {p1, p2}}
]
(p1) circle (2pt) node[below left]  {$-\delta_1+2\delta_2$}
(p2) circle (2pt) node[above right] {$\delta_1-2\delta_2$};

\fill[
name intersections={of=X and B, by = {p3, p4}}
]
(p3) circle (2pt) node[below left]  {$\alpha_2$}
(p4) circle (2pt) node[below right] {$\alpha_1$};

\fill[
name intersections={of=Y and A, by = {p5}}
]
(p5) circle (2pt) node[left] {$\delta_1-2\delta_2$};

\fill[
name intersections={of=Y and B, by = {p6}}
]
(p6) circle (2pt) node[below left]  {$-2 \delta_2 e^{-\tau}$};

\path[
name intersections={of=X and Y, by = {p7, p8}}
];

\addplot[
gray!65,
fill opacity = 0.85
]
fill between[
of=A and B,
soft clip={ 
(-0.1689,0) rectangle (0.1689,0.33)}
];


\fill[gray!65, opacity = 0.85
] (axis cs:-0.1689,0.1689)--(axis cs:-0.168,0.168)--(axis cs:-0.07,0.01)--(axis cs:-0.06,0)--(axis cs:0.06,0)--(axis cs:0,0.05)--cycle;


\addlegendentry{
$\xi= \delta_1 - 2\delta_2 - |\alpha|$
}
\addlegendentry{
$\xi= \frac{2 \delta_1 \ell^2}{T} \alpha^2 + |\alpha| - 2 e^{-\tau} \delta_2 $
}
\end{axis}
\end{tikzpicture}
\end{center}
\caption{The exponential stability region $\Sigma$}
\label{fig}
\end{figure}
\item The stability region $\Sigma$ is symmetric with respect to the $\xi$-axis, as its boundaries are defined by even functions. Furthermore, since the right-hand side of the second inequality in \eqref{eq:112} is quadratic with respect to $\alpha$, it yields two symmetric roots $\alpha_1, \alpha_2$ given by:
\begin{gather}  \label{eq:115}
\alpha_1 = \frac{T}{4 \delta_1 \ell^2} \left[ -1+ 
\left( 1 + \frac{16 \delta_1 \delta_2 e^{-\tau} \ell^2}{T}
\right)^{\frac{1}{2}}
\right],
\qquad \alpha_2 = - \alpha_1.
\end{gather}
\item Note that as $\kappa \to \infty$ or $\kappa \to 0$, the constants $\delta_1$ and $\delta_2$ vanish, which consequently implies that $\alpha$ tends to zero as well.
\end{enumerate}
\end{remark}

\begin{remark}
A relationship between $\alpha$ and $\kappa$ can be obtained by fixing $\xi$ and $\delta_i$. Given that $\delta_1$ and $\delta_2$ must satisfy the conditions \eqref{eq:110}, we adopt the following choices:
\begin{equation} \label{eq:116}
\begin{gathered}
\delta_1 = \min\left\{ \left( \max \Big\{ \ell; \frac{\ell}{T}
\Big\} \right)^{-1};  
\kappa \bigg( \frac{\ell}{2} + \frac{\kappa^2 \ell^3}{3 - T\ell^2}
\bigg)^{-1}
\right\},
\\ 
\delta_2 = \frac{\delta_1}{4} \; \text{ and } \; \xi = 2|\alpha|.
\end{gathered}
\end{equation}
Consequently, the estimate \eqref{eq:111} holds, and the stability region $\Sigma$ becomes:
\begin{gather}  \label{eq:117}
\Sigma := \left\{
\big(\alpha,2|\alpha|\big): \; 
|\alpha| < \frac{\delta_1}{6}, \; 
-\dfrac{2 \delta_1 \ell^2}{T} \alpha^2 + |\alpha| + \frac{1}{2} e^{-\tau} \delta_1 \geqslant 0
\right\}, \quad \alpha \in \mathbb{R}^\ast.
\end{gather}
Let $\Psi(\alpha) := -\frac{2 \delta_1 \ell^2}{T} \alpha^2 + |\alpha| + \frac{1}{2} e^{-\tau} \delta_1$, where the parameter $\delta_1$ is defined in \eqref{eq:116}. 
Since $\Psi$ is an even function, the equation $\Psi(\alpha) = 0$ yields two symmetric roots explicitly given by:
\begin{gather}  \label{eq:118}
\alpha_1^\ast 
= \frac{T}{4 \delta_1 \ell^2} \bigg[ 1 + \left( 1 + \frac{4 e^{-\tau} \delta_1^2 \ell^2}{T} 
\right)^{\frac{1}{2}} 
\bigg], \quad 
\alpha_2^\ast = - \alpha_1^\ast.
\end{gather}
The conditions \eqref{eq:110} are satisfied for all $|\alpha| < \alpha_0$, with $\alpha_0$ defined as:
\begin{gather}  \label{eq:119}
\alpha_0 = \min\left\{ \frac{\delta_1}{6}; \, \alpha_1^\ast 
\right\},
\end{gather}
in which $\delta_1$ is given by \eqref{eq:116}. Therefore, $\alpha_0$ depends exclusively on the geometry of the beam, the internal time delay $\tau$ and the feedback gain $\kappa$.
\end{remark}


\section{Conclusion}

In this paper, we have established the exponential stabilization of an Euler-Bernoulli beam subjected to a constant axial force and a fixed internal time delay, whose coefficient is sign-indefinite, under the action of a boundary velocity control.
Our approach first consisted in reformulating the system \eqref{eq:p1} as an abstract Cauchy problem within an appropriate Hilbert space, in order to prove its well-posedness via the L\"umer-Phillips Theorem. 
Subsequently, we constructed a Lyapunov functional based on the system's energy. Its decay along the trajectories yields an $(\alpha, \xi)$-region of exponential stability for system \eqref{eq:p1}, defined by the set $\Sigma$ and symmetric with respect to the $\xi-$axis.
For future research, we intend to extend these results to the case of time-varying delays. We will also explore the incorporation of an internal damping mechanism within the model to further enhance its exponential stability.

\bibliographystyle{unsrt}
\bibliography{main}

\end{document}